\documentclass[a4paper,11pt]{amsart}

\let\oldtocsection=\tocsection
 
\let\oldtocsubsection=\tocsubsection
 
\let\oldtocsubsubsection=\tocsubsubsection
 
\renewcommand{\tocsection}[2]{\hspace{0em}\oldtocsection{#1}{#2}\bfseries}
\renewcommand{\tocsubsection}[2]{\hspace{1.8em}\oldtocsubsection{#1}{#2}}
\renewcommand{\tocsubsubsection}[2]{\hspace{4.4em}\oldtocsubsubsection{#1}{#2}}

\makeatletter
\renewcommand\subsection{\@startsection{subsection}{2}%
  \z@{-.5\linespacing\@plus-.7\linespacing}{.5\linespacing}%
  {\normalfont\scshape}}
\renewcommand\subsubsection{\@startsection{subsubsection}{3}%
  \z@{.5\linespacing\@plus.7\linespacing}{.5\linespacing}%
  {\normalfont\scshape}}
\makeatother

\usepackage[framemethod=tikz]{mdframed}

\usepackage[utf8]{inputenc}
\usepackage[T1]{fontenc}
\usepackage{tikz-cd}
\usepackage{paralist}
\usepackage{mathrsfs}
\usepackage{amssymb}
\usepackage{amsmath}
\usepackage{amsfonts}

\newtheorem{theorem}{Theorem}[section]
\newtheorem{corollary}[theorem]{Corollary}
\newtheorem{lemma}[theorem]{Lemma}

\newtheorem{Fact}[theorem]{Fact}
\newtheorem{claim}[theorem]{Claim}

\theoremstyle{definition}
\newtheorem{definition}[theorem]{Definition}
\newtheorem{remark}[theorem]{Remark}

\usepackage[pdftex]{hyperref}
\hypersetup{
    colorlinks=true, 
    linktoc=all,     
    linkcolor=blue,  
    allcolors=blue,
}

\newcommand{\dom}[1]{\ensuremath{\mathrm{dom}}(#1)}

\newcommand{\power}{\ensuremath{\mathscr{P}}}
\newcommand{\set}[2]{\ensuremath{\{#1 \,|\, #2 \}}}
\newcommand{\seq}[2]{\ensuremath{\langle #1 \,|\, #2 \rangle}}

\renewcommand{\iff}{\leftrightarrow}

\newcommand{\el}{\prec}

\newcommand{\sub}{\subseteq}

\newcommand{\la}{\langle}
\newcommand{\ra}{\rangle}

\newcommand{\bb}{\mathbb}

\newcommand{\beq}{\begin{equation}}
\newcommand{\eeq}{\end{equation}}
\newcommand{\brm}{\begin{remark}\begin{rm}}
\newcommand{\erm}{\end{rm}\end{remark}}
\newcommand{\mx}{\mathrm}
\newcommand{\bce}{\begin{compactenum}}
\newcommand{\ece}{\end{compactenum}}

\newcommand{\cf}{\mathrm{cf}}

\newcommand{\Add}{\mathrm{Add}}

\newcommand{\Q}{\bb{Q}}
\renewcommand{\P}{\bb{P}}

\newcommand{\M}{\bb{M}}

\newcommand{\x}{\times}

\newcommand{\TP}{{\sf TP}}

\newcommand{\ISP}{{\sf ISP}}

\newcommand{\ZFC}{\sf ZFC}

\newcommand{\CH}{\sf CH}

\newcommand{\PFA}{\sf PFA}
\newcommand{\MA}{\sf MA}

\newcommand{\wKH}{\sf wKH}

\newcommand{\GMP}{\mathsf{GMP}}
\newcommand{\KH}{\mathsf{KH}}
\newcommand{\CC}{\mathsf{CC}}

\newcommand{\RR}{\mathcal{R}}
\newcommand{\se}[1]{\left\{#1\right\}}
\newcommand{\seqv}[1]{\left\langle #1\right\rangle}

\newcommand{\ST}{S(\dot{T})}

\begin{document}

\title[Indestructibility of some compactness principles \ldots]{Indestructibility of some compactness principles over models of $\PFA$}

\author{Radek Honzik}
\address[Honzik]{
Charles University, Department of Logic,
Celetn{\' a} 20, Prague~1, 
116 42, Czech Republic
}
\email{radek.honzik@ff.cuni.cz}
\urladdr{logika.ff.cuni.cz/radek}

\author{Chris Lambie-Hanson}
\address[Lambie-Hanson]{
Institute of Mathematics, 
Czech Academy of Sciences, 
{\v Z}itn{\'a} 25, Prague 1, 
115 67, Czech Republic
}
\email{lambiehanson@math.cas.cz}
\urladdr{https://users.math.cas.cz/~lambiehanson/}

\author{{\v S}{\'a}rka Stejskalov{\'a}}
\address[Stejskalov{\'a}]{
Charles University, Department of Logic,
Celetn{\' a} 20, Prague~1, 
116 42, Czech Republic
}
\email{sarka.stejskalova@ff.cuni.cz}
\urladdr{logika.ff.cuni.cz/sarka}

\address{Institute of Mathematics, Czech Academy of Sciences, {\v Z}itn{\'a} 25, Prague 1, 115 67, Czech Republic}

\thanks{
R.~Honzik and {\v S}.~Stejskalov{\'a} were supported by FWF/GA{\v C}R grant \emph{Compactness principles and combinatorics} (19-29633L)}

\begin{abstract}
We show that $\PFA$ (Proper Forcing Axiom) implies that adding any number of Cohen subsets of $\omega$ will not add an $\omega_2$-Aronszajn tree or a weak $\omega_1$-Kurepa tree, and moreover no $\sigma$-centered forcing can add a weak $\omega_1$-Kurepa tree (a tree of height and size $\omega_1$ with at least $\omega_2$ cofinal branches). This partially answers an open problem whether ccc forcings can add $\omega_2$-Aronszajn or $\omega_1$-Kurepa trees.

We actually prove more: We show that a consequence of $\PFA$, namely the \emph{guessing model principle}, $\GMP$, which is equivalent to the \emph{ineffable slender tree property}, $\ISP$, is preserved by adding any number of Cohen subsets of $\omega$. And moreover, $\GMP$ implies that no $\sigma$-centered forcing can add a weak $\omega_1$-Kurepa tree (see Section \ref{sec:guess_def} for definitions).

For more generality, we study the principle $\GMP$ at an arbitrary regular cardinal $\kappa = \kappa^{<\kappa}$ (we denote this principle $\GMP_{\kappa^{++}}$), and as an application we show that there is a model in which there are no weak $\aleph_{\omega+1}$-Kurepa trees and no $\aleph_{\omega+2}$-Aronszajn trees.
\end{abstract}

\keywords{$\PFA$; the tree property; weak Kurepa Hypothesis; indestructibility; guessing models}
	\subjclass[2010]{03E55, 03E35}
	\maketitle

\tableofcontents

\section{Introduction}

It has been a question of some interest whether ``small'' forcings (either in terms of size or chain condition) can add ``large'' trees. Depending on the meaning of ``small'' and ``large'', there are both negative and positive results. Here are some examples for the positive answer:
 
\begin{itemize}
\item By a result of Shelah, a single Cohen subset of $\omega$ adds an $\omega_1$-Suslin tree. 
\item Rinot showed in \cite{Rinot:Addsp.Tree} that a cofinality preserving forcing of size $\omega_3$ can add a special Aronszajn tree at $\aleph_{\omega_1+1}$ (this requires large cardinals).
\item Jin and Shelah \cite{JS:Kurepa} showed that an $\omega_1$-distributive forcing of size $\omega_1$ can add an $\omega_1$-Kurepa tree. 
\end{itemize}

However, if we take ``small'' to be countable or just ccc, and ``large'' to mean an $\omega_2$-Aronszajn tree or a (weak) $\omega_1$-Kurepa tree, the question is open. Even the simplest question whether a single Cohen subset of $\omega$ can add an $\omega_2$-Aronszajn tree or an $\omega_1$-Kurepa tree over some model of $\ZFC$ remains unanswered. Similar cases can be considered at larger cardinals as well (see for instance \cite{HM:dest} for $\aleph_{\omega+1}$).

For the negative answer, more results are known. The main reason is that many principles inherited from large cardinals (``compactness principles'') tend to prohibit the existence of certain trees, and by starting with a carefully chosen model where such principles hold, one can sometimes show that no trees of the given type are added by small forcings. But as an answer to the question above, this approach may appear unconvincing, inasmuch as it depends on the particular model in question (see a brief summary of these results in Section \ref{sec:back}).

A more convincing approach is to find an assumption $\varphi$ such that $\ZFC + \varphi$ prohibits the existence of certain trees, and $\varphi$ itself is always preserved by small forcings, not just over some particular model under consideration.\footnote{$\varphi$ can be just the sentence that there are no trees of the given type, but in the known examples, a stronger principle is usually required.} For instance Chang's Conjecture, $\CC$, plays this role for $\omega_1$-Kurepa trees: by a well-known theorem, $\CC$ is preserved by all ccc forcings, and hence no ccc forcing can add an $\omega_1$-Kurepa tree over any model of $\CC$. If we subscribe to $\ZFC + \CC$, then we conclude that no ccc forcing can add an $\omega_1$-Kurepa tree.

In this paper we show that the \emph{guessing model principle}, $\GMP$,\footnote{We generalize this principle to larger cardinals, so we will have $\GMP = \GMP_{\omega_2}$ in what follows.} which is a consequence of $\PFA$, plays a similar role for $\omega_2$-Aronszajn trees and weak $\omega_1$-Kurepa trees (see Section \ref{sec:back} for definitions, and Lemma \ref{lm:GMPimplies}, Corollary \ref{cor:GMP} and Theorem \ref{th:2} for proofs): 

\medskip

\textbf{Theorem.} {\it $\ZFC + \GMP$ proves that adding any number of Cohen subsets of $\omega$ will not add an $\omega_2$-Aronszajn tree or a weak $\omega_1$-Kurepa tree. Moreover, it 
proves that no $\sigma$-centered forcing can add a weak $\omega_1$-Kurepa tree.}

\medskip

$\GMP$ is a certain compactness principle introduced by Viale and Weiss in \cite{VW:PFA}. It follows from $\PFA$, but it is strictly weaker: for instance, it does not put any bound on the value of $2^\omega$, apart from contradicting $\CH$. This principle can either be formulated in terms of guessing models, or equivalently in terms of slender lists (generalizations of trees), and is known to capture the ``combinatorial core'' of supercompactness; see \cite{Weiss:gentree} and \cite{VW:PFA} for more  details.  For our proof, we find it more convenient to work with guessing models.

The paper is structured as follows. In Section \ref{sec:back} we review the basic definitions and provide a brief survey of results related to preservation of the tree property (i.e.\ not adding Aronszajn trees) or the negation of the weak Kurepa Hypothesis (i.e.\ not adding weak Kurepa trees). 

In Section \ref{sec:GMP}, generalizing the principle $\GMP$ from \cite{VW:PFA}, we define for a regular $\kappa^{<\kappa} = \kappa$ a principle $\GMP_{\kappa^{++}}$ (with $\GMP = \GMP_{\omega_2}$) and show that it is preserved by adding any number of Cohen subsets of $\kappa$ (Corollary \ref{cor:GMP}). Since $\GMP_{\kappa^{++}}$ implies that there are no $\kappa^{++}$-Aronszajn trees and no weak $\kappa^+$-Kurepa trees (Lemma \ref{lm:GMPimplies}), we obtain the desired result. 

In Section \ref{sec:centered}, Theorem \ref{th:2}, we prove a slightly stronger result that over models of $\GMP_{\kappa^{++}}$, no $\kappa^+$-centered forcing (see Definition \ref{def:centered}) can add a weak $\kappa^+$-Kurepa tree.

In Section \ref{sec:appl}, we give an application of this result, showing that starting with some large cardinals there is a model in which there are no weak $\aleph_{\omega+1}$-Kurepa trees and no $\aleph_{\omega+2}$-Aronszajn trees.

In Section \ref{sec:open} we state some open questions.

\subsection{Preliminaries}\label{sec:back}

Suppose $\lambda$ is a regular cardinal. We say that the \emph{tree property} at $\lambda$, $\TP(\lambda)$, holds if every $\lambda$-tree has a cofinal branch; equivalently, there are no $\lambda$-Aronszajn trees. We say that the \emph{Kurepa Hypothesis} at $\lambda$, $\KH(\lambda)$, holds if there is a $\lambda$-tree with at least $\lambda^+$-many cofinal branches (we call such a tree a $\lambda$-Kurepa tree); we say that the \emph{weak Kurepa Hypothesis} at $\lambda$, $\wKH(\lambda)$,  holds if there is a tree of height and size $\lambda$ with at least $\lambda^+$-many cofinal branches (we call such a tree a weak $\lambda$-Kurepa tree). A $\lambda$-Aronszajn tree is an incompact object because it has chains of every size $<\lambda$, but no chains of size $\lambda$; similarly a (weak) $\lambda$-Kurepa tree is an incompact object because every level of the tree has size $<\lambda$ (or $\le \lambda)$, yet there are $\lambda^+$-many cofinal branches. 

It is known that together with inaccessibility, these properties are related to large cardinals (see for instance Devlin \cite{DEVbook}):

\begin{Fact}
\begin{enumerate}[(i)] Suppose $\lambda$ is an inaccessible cardinal.
\item $\lambda$ is weakly compact if and only if $\TP(\lambda)$.
\item If $\lambda$ is ineffable, then $\neg \KH(\lambda)$.\footnote{If $\lambda$ is inaccessible, we say that $T$ is a $\lambda$-tree if for every $\omega \le \alpha < \lambda$, $|T_\alpha|\le |\alpha|$ (sometimes such trees are called \emph{slim} $\lambda$-trees). This prevents the full binary tree $2^{<\lambda}$ from being a witness for $\KH(\lambda)$. However $\wKH(\lambda)$ is always true for an inaccessible $\lambda$, with $2^{<\lambda}$ being a witness.} Moreover, if $V=L$, then the converse is true as well.
\end{enumerate}
\end{Fact}

Notice that the characterization of weak compactness by $\TP$ is provable in $\ZFC$, while the characterization of ineffability by $\neg \KH$ requires $V = L$ (to our knowledge, it is open whether the assumption $V =L$ can be removed).

Unlike the inaccessibility of $\lambda$, the principles $\TP(\lambda)$ and $\neg \KH(\lambda)$ are more robust, and may be viewed as the ``combinatorial core'' of the respective large cardinal notions. For instance adding $\lambda$ many Cohen subsets of $\omega$ to an ineffable cardinal $\lambda$ will destroy the inaccessibility of $\lambda$, but will not add any $\lambda$-Aronszajn trees or $\lambda$-Kurepa trees.

Silver in \cite{Silver:KH} was the first one to show that the non-existence of Kurepa trees can also hold on a successor cardinal: using the Levy collapse of an inaccessible, he showed that $\neg \KH(\omega_1)$ is consistent (with $\CH$). Soon afterwards, Mitchell, in \cite{M:tree}, devised a different collapse which simultaneously adds subsets of $\omega$, and showed that $\TP(\omega_2)$ and $\neg \wKH(\omega_1)$ are consistent; see Abraham's \cite{ABR:tree} for the now-standard presentation of the forcing. We shall denote the forcing devised by Mitchell by $\M(\kappa,\lambda)$, where it is understood that $\kappa$ is a regular cardinal satisfying $\kappa^{<\kappa} = \kappa$ and $\lambda$ is (at least) an inaccessible cardinal. $\M(\kappa,\lambda)$ collapses cardinals in the open interval $(\kappa,\lambda)$ and turns $\lambda$ into $\kappa^{++}$. By \cite{M:tree} or \cite{ABR:tree} if $\lambda$ is weakly compact, then $\TP(\kappa^{++})$ is true in $V[\M(\kappa,\lambda)]$. A similar argument shows that if $\lambda$ is inaccessible, then $\neg \KH(\kappa^{+})$ and in fact $\neg \wKH(\kappa^{+})$ hold in $V[\M(\kappa,\lambda)]$.\footnote{Notice the apparent disparity between the parameters in these two principles after the collapse: while the tree property is still considered at $\lambda$, the negation of the (weak) Kurepa hypothesis has $\kappa^+$ as the parameter (the large cardinal $\lambda$ is used to control the number of cofinal branches of trees of size $\kappa^+$).} 

With the discovery that many compactness principles can hold at successor cardinals, it was natural to inquire which forcing notions can destroy them and which will preserve them: the motivation being a general interest, and also an interest in developing a technical tool for forcing constructions. Todorcevic showed the compatibility of $\MA_{\omega_1}$ with $\neg \wKH(\omega_1)$ in \cite{T:wKH}, and essentially proved that over the Mitchell model $V[\M(\omega,\lambda)]$, $\lambda$ inaccessible, a finite support iteration of length $\omega_2$ of ccc forcings of size at most $\omega_1$ which do not add cofinal branches to $\omega_1$-Suslin trees does not add weak $\omega_1$-Kurepa trees; by Remark 1.7 in the same paper, it follows that the same iteration does not add $\omega_2$-Aronszajn trees over $V[\M(\omega,\lambda)]$, $\lambda$ weakly compact. Unger studied the indestructibility over the Mitchell model explicitly in his \cite{UNGER:1}, and further results appeared in \cite{HS:ind} and \cite{HS:u}.  Apart from Todorcevic's \cite{T:wKH}, the preservation of $\neg \wKH(\omega_1)$ over the Mitchell model has not been studied explicitly,\footnote{There is a preservation result for $\neg \KH(\omega_1)$ over the Levy collapse, see \cite{JS:k}. This however does not fit with our topic here because in that model $\CH$ is true, so necessarily both $\TP(\omega_2)$ and $\neg \wKH(\omega_1)$ fail.} but it is likely that the analogues of the results for the tree property obtained in \cite{HS:ind} also apply to the negation of the weak Kurepa Hypothesis. One can extend this line of inquiry to consider other variants of Mitchell forcing, and also other methods for obtaining $\TP(\omega_2)$ and $\neg \wKH(\omega_1)$, for instance the Sacks forcing (see \cite{KANAMORIperfect}).

In all these examples, however, the question of preservation of the compactness principles is asked over specific models, and the strategy of proofs follows the same pattern: Since we know the ``history'' of the cardinal $\lambda$ in our model, we can return back to the ground model where $\lambda$ is a large cardinal, and we can fix an appropriate elementary embedding and do a lifting argument. For instance to argue that $\P$ preserves $\TP(\omega_2)$ over the model $V[\M(\omega,\lambda)]$, the typical argument is to lift an elementary embedding with critical point $\lambda$ to the forcing $\M(\omega,\lambda) *\dot{\P}$. But what if do not have an elementary embedding, and all we know is that $\TP(\omega_2)$ or $\neg \wKH(\omega_1)$ are true: can we say something about the preservation?

This leads to the question of considering preservation of compactness principles over theories extending $\ZFC$, not just models. Here are some known results in this direction (see Definition \ref{def:centered} for the definition of $\kappa^+$-centeredness): 

\begin{itemize}
\item Chang's Conjecture is preserved by all ccc forcings; hence, over models of Chang's Conjecture, so is its consequence $\neg \KH(\omega_1)$.
\item Foreman showed in \cite{Foreman:sat} that $\kappa^{++}$-saturated ideals over $\kappa^{+}$, $\kappa$ regular, are preserved by $\kappa^+$-centered forcing notions (in the sense that they generate saturated ideals in the extension).
\item Gitik and Krueger showed in \cite{GK:a} that the negation of the approachability property at $\kappa^{++}$, $\kappa$ regular, is preserved by all $\kappa^+$-centered forcings. 
\item The first and the third author of the present paper showed in \cite{HS:u} that stationary reflection at $\kappa^+$, $\kappa$ regular, is preserved by all $\kappa$-cc forcing notions. They further showed in \cite{HS:u} that if $\kappa^{<\kappa}$, then club stationary reflection at $\kappa^{++}$ is preserved by Cohen forcing at $\kappa$ and Prikry forcing at $\kappa$.\footnote{This preservation result has been recently extended to all $\kappa^+$-linked forcings in \cite{TS:CSR} (being $\kappa^+$-linked is slightly weaker than $\kappa^+$-centered).}
\end{itemize}

We extend this list in this paper by showing (see Definition \ref{guessing_model_def} and Corollary \ref{cor:GMP}):

\begin{itemize}
\item $\GMP_{\kappa^{++}}$ is preserved by adding any number of Cohen subsets of $\kappa$, and hence over models of $\GMP_{\kappa^{++}}$, so are its consequences $\TP(\kappa^{++})$ and $\neg \wKH(\kappa^+)$.
\end{itemize}

\section{Preservation of the Guessing Model Principle by Cohen forcing}\label{sec:GMP}

\subsection{Guessing models}\label{sec:guess_def}

In \cite{Weiss:gentree}, building on work of Jech \cite{jech:tree} and Magidor \cite{Magidor:supercompact} providing combinatorial characterizations of strongly compact and supercompact cardinals, Weiss introduced the notion of a \emph{slender $\power_\kappa (\lambda)$-list} and used this to formulate a powerful compactness principle called the \emph{ineffable slender tree property}, $\ISP_\kappa$. For inaccessible $\kappa$, $\ISP_\kappa$ characterizes supercompactness and can be seen as capturing the combinatorial core of
supercompactness, but $\ISP_\kappa$ can also consistently hold at an accessible $\kappa$. Viale and Weiss \cite{VW:PFA} proved that $\PFA$ implies $\ISP_{\omega_2}$, and moreover provided an equivalent characterization of $\ISP_{\omega_2}$ in terms of \emph{guessing models}.

In this paper, we work with this equivalent characterization in terms of guessing models and its generalizations. We state the definition in the form relevant for us; even more general versions are possible (see Remark \ref{rm:more}). Recall the notation $\power_\mu(x)$ which denotes the set of all subsets of $x$ of size $<\mu$.

\begin{definition} \label{guessing_model_def}
Let $\kappa^{<\kappa} = \kappa < \kappa^{++} \le \theta$ be regular cardinals. Let $M \el H(\theta)$ be an elementary submodel of size $\kappa^+$ satisfying ${}^{<\kappa}M \sub M$.
  \begin{enumerate}
    \item Given a set $x \in M$, and a subset $d \subseteq x$, we say that 
    \begin{enumerate}
      \item $d$ is \emph{$(\kappa^+, M)$-approximated} if, for every $z \in M \cap \power_{\kappa^+}(M)$, 
      we have $d \cap z \in M$;
      \item $d$ is \emph{$M$-guessed} if there is $e \in M$ such that $d \cap M = e \cap M$.
    \end{enumerate}
    \item For $x \in M$, $M$ is a \emph{$\kappa^+$-guessing model for $x$} if
    every $(\kappa^+, M)$-approximated subset of $x$ is $M$-guessed.
    \item $M$ is a \emph{$\kappa^+$-guessing model} if, for every $x \in M$, it is a $\kappa^+$-guessing 
    model for $x$.
  \end{enumerate}
We denote by $\GMP_{\kappa^{++}}(\theta)$ the assertion that the set of $M \in \power_{\kappa^{++}}(H(\theta))$ such that $M$ is a $\kappa^+$-guessing model is stationary in $\power_{\kappa^{++}}(H(\theta))$. We write $\GMP_{\kappa^{++}}$ if $\GMP_{\kappa^{++}}(\theta)$ holds for every regular $\theta \ge \kappa^{++}$.
\end{definition}

Note that we specifically require that the guessing models are closed under sequences of length $<\kappa$. In other words, we require in $\GMP_{\kappa^{++}}(\theta)$ that the guessing models concentrate on the stationary set of all $x \in \power_{\kappa^{++}}(H(\theta))$ such that ${}^{<\kappa}x \sub x$ (this set is in fact closed under increasing unions of cofinality at least $\kappa$). For the principle $\GMP_{\omega_2}$ from \cite{VW:PFA}, this is automatic because it just means that the models are closed under finite sequences; however for $\kappa > \omega$ we need to require this property to show that the concept behaves as expected (see for instance Lemma \ref{lm:gprop}).

Viale and Weiss essentially proved in \cite{VW:PFA} the following (see \cite{VW:PFA} for the definition of $\ISP_{\omega_2}$):

\begin{Fact}\label{PFA}
\begin{enumerate}[(i)]
\item $\GMP_{\omega_2}$ is equivalent to $\ISP_{\omega_2}$.
\item $\PFA$ implies $\GMP_{\omega_2}$.
\end{enumerate}
\end{Fact}

Models with $\GMP_{\kappa^{++}}$ can be obtained starting with sufficiently large cardinals:

\begin{Fact}\label{f:GMP}
Suppose $\kappa = \kappa^{<\kappa}$ is regular and $\lambda>\kappa$ is supercompact. Then in the Mitchell model $V[\M(\kappa,\lambda)]$, which turns $\lambda$ to $\kappa^{++}$, $\GMP_{\kappa^{++}}$ holds.
\end{Fact}

\begin{proof}
The argument is essentially the same as in \cite{Weiss:gentree} and \cite{VW:PFA}. Let us only mention why we can assume that the guessing models concentrate on sets closed under $<\kappa$-sequences, as we require in our definition of $\GMP_{\kappa^{++}}$. This follows from the fact that the characterization of $\theta$-supercompactness of a supercompact cardinal $\lambda$ by means of $\theta$-ineffability (see \cite[p.~281]{Magidor:supercompact}) ensures the stationarity of the required set by showing the stronger property, namely that the set in question is an element of a normal ultrafilter on  $\power_\lambda(\theta)$. When reformulated for $\power_{\lambda}(H(\theta))$, with $\theta = |H(\theta)|$, it follows we can start with stationary sets which concentrate on the set of all submodels of $H(\theta)$ which are closed under sequences of length $<\mu$, for any fixed $\mu < \lambda$. Since $\M(\kappa,\lambda)$ does not add new sequences of length $<\kappa$ (but adds many new subsets of $\kappa$), the argument proceeds with $\mu = \kappa$.
\end{proof}

The following lemma generalizes the analogous lemma in \cite{CK:ind} which was formulated for $\kappa = \omega$.

\begin{lemma} \label{lm:gprop}Suppose $\kappa$ is an infinite regular cardinal, and let $M \el H(\theta)$ be a $\kappa^+$-guessing model for some regular $\theta \ge \kappa^{++}$. Let $\nu \in M$ be a cardinal with cofinality $\ge \kappa^+$.
\bce[(i)]
\item Then $\cf(\mx{sup}(M \cap \nu)) = \kappa^+$, and in particular $\kappa^+ \sub M$ and so $M \cap \kappa^{++}$ is an ordinal.
\item If $\theta \ge \kappa^{+3}$, then moreover $\cf(M \cap \kappa^{++}) = \kappa^+$.
\ece
\end{lemma}

\begin{proof}
(i). This is like Lemma 2.3 from \cite{CK:ind} using the fact that ${}^{<\kappa}M \sub M$.

(ii). If $\theta \ge \kappa^{+3}$, then $\kappa^{++} \in M$, and the claim follows by (i).
\end{proof}

Generalizing the results known for $\GMP_{\omega_2}$, let us review the argument that $\GMP_{\kappa^{++}}$ implies $\TP(\kappa^{++})$ and $\neg \wKH(\kappa^+)$.

\begin{lemma}\label{lm:GMPimplies}
$\GMP_{\kappa^{++}}$ implies $\TP(\kappa^{++})$ and $\neg \wKH(\kappa^+)$, and hence in particular $2^\kappa \ge \kappa^{++}$.
\end{lemma}

\begin{proof}
Let us first show that $\TP(\kappa^{++})$ holds. Suppose $T$ is a $\kappa^{++}$-tree; we wish to show that it contains a cofinal branch. Choose a $\kappa^+$-guessing model $M \el H(\kappa^{+3})$ of size $\kappa^+$, with $\kappa^+ \in M$ and $T \in M$. Let $t$ be any node in $T$ on level $\delta = M \cap \kappa^{++}$. Denote by $d = \set{s}{s <_T t}$ the set of predecessors of $t$ in $T$; notice that $d \sub M$, and $d \sub T \in M$. The set $d$ is $(\kappa^+,M)$-approximated: if $z \in M$ has size $\kappa$, then since $\delta$ has cofinality $\kappa^+$ (Lemma \ref{lm:gprop}(i)), there is $t^* \in d$ such that $z \cap d$ is definable in $M$ as the set of all predecessors of $t^*$ which are in $z$, and so $z \cap d \in M$. Since $M$ is a guessing model, and we showed that $d$ is $(\kappa^+,M)$-approximated, there is $e \in M$ with $e \cap M = d \cap M$. It follows $M \models (\mbox{$e$ is a cofinal branch in $T$})$, and by elementarity this is true in $H(\kappa^{+3})$, and hence in $V$. 

In a similar way we can show that $\neg \wKH(\kappa^+)$ holds. Suppose for contradiction that $T$ is a tree of size $\kappa^+$ which has more than $\kappa^+$-many branches. Choose a $\kappa^+$-guessing model $M \el H(\kappa^{++})$ of size $\kappa^+$, with $\kappa^+ \sub M$ and $T \in M$. Every cofinal branch $b \sub T$ is $(\kappa^+,M)$-approximated, using the fact that $T$ has height $\kappa^+$. It follows that every cofinal branch $b$ must be in $M$, but this contradicts the fact that $M$ has size $\kappa^+$.
\end{proof}

Note that we need just the principle $\GMP_{\kappa^{++}}(\kappa^{+3})$ for obtaining $\TP(\kappa^{++})$, and $\GMP_{\kappa^{++}}(\kappa^{++})$ for obtaining $\neg \wKH(\kappa^+)$.

\brm \label{rm:more}
One may consider $(\mu,M)$-approximations for $\mu < \kappa^+$ (with the obvious modification of the definition in (1)(a) above), dropping the assumption $\kappa^{<\kappa} = \kappa$,  which leads to a three-parameter principle denoted $\GMP(\mu,\kappa^{++},\theta)$,
with our $\GMP_{\kappa++}(\theta)$ being $\GMP(\kappa^+,\kappa^{++},\theta)$ together with the requirement that our guessing models be closed under sequences of length $<\kappa$. See \cite{ChS:guess} for more details regarding the more general notion of guessing.
\erm

\subsection{A preservation theorem}

In \cite{cox_krueger}, Cox and Krueger prove that $\GMP_{\omega_2}$ is compatible with any consistent value of the continuum greater than $\omega_1$ by producing a specific model of $\ZFC$ over which $\GMP_{\omega_2}$ is indestructible under adding any number of Cohen reals. In this section, we remove the dependence of their result on a particular choice of ground model by proving that, over \emph{any} model of $\ZFC$, if $\kappa^{<\kappa} = \kappa$, then $\GMP_{\kappa^{++}}$ is preserved by adding any number of Cohen subsets of $\kappa$.\footnote{We would like to thank to Menachem Magidor who shared with us his unpublished proof that $\GMP_{\omega_2}$ is preserved by adding a single Cohen subset of $\omega$. We have generalized his result to adding any number of Cohen subsets of $\kappa$.}

The following lemma is essentially due to Krueger \cite{krueger_sch}, presented here in a 
slightly more general form.

\begin{lemma} \label{unbounded_prop}
  Suppose that $\kappa < \theta$ are infinite regular cardinals and 
  $M \prec H(\theta)$ is a $\kappa^+$-guessing model such that $\kappa^+ \subseteq M$ and 
  ${^{<\kappa}}M \subseteq M$. Then $M$ is $\kappa^+$-internally unbounded, 
  i.e., for every $z \in M$ and every $x \in [z]^\kappa$, there is 
  $y \in [z]^\kappa \cap M$ such that $x \subseteq y$.
\end{lemma}

\begin{proof}
  Assume for sake of contradiction that there is $z \in M$ and $x \in [z]^\kappa$ 
  such that there is no $y \in [z]^\kappa \cap M$ for which $x \subseteq y$. Injectively 
  enumerate $x$ as $\langle a_\alpha \mid \alpha < \kappa \rangle$ and, for each 
  $\beta < \kappa$, let $x_\beta := \{ a_\alpha \mid \alpha < \beta \}$.
  Since ${^{<\kappa}}M \subseteq M$, we have $x_\beta \in M$ for all $\beta < \kappa$.
  Let $x^* := \{x_\beta \mid \beta < \kappa\}$. Then $x^* \subseteq [z]^{<\kappa} \in M$.
  
  \begin{claim}
    $x^*$ is $(\kappa^+, M)$-approximated.
  \end{claim}
  
  \begin{proof}
    Fix $w \in M$ with $|w| \leq \kappa$. Since $x^* \subseteq [z]^{<\kappa} \in M$, 
    we can assume that $w \subseteq [z]^{<\kappa}$. If it were the case that 
    $|w \cap x^*| = \kappa$, then $\bigcup w$ would be an element of $[z]^{\kappa} \cap M$
    covering $x$, contradicting our assumption. Therefore, $|w \cap x^*| < \kappa$, 
    so, since ${^{<\kappa}}M \subseteq M$, we have $w \cap x^* \in M$.
  \end{proof}
  
  Since $M$ is a $\kappa^+$-guessing model, we can find $e \in M$ such that 
  $e \cap M = x^* \cap M = x^*$. If it were the case that $|e| > \kappa$, then there would 
  be an injection $f: \kappa^+ \rightarrow e$ with $f \in M$, and since $\kappa^+ 
  \subseteq M$, we would have $|e \cap M| > \kappa$, contradicting the fact that 
  $e \cap M = x^*$ and $|x^*| = \kappa$. Therefore, $|e| = \kappa$. Since $\kappa 
   \subseteq M$, it follows that $e \subseteq M$, and therefore we in fact have $e = x^*$.
  It then follows that $\bigcup e = x$ is an element of $M$, again contradicting our assumption 
  and completing the proof.
\end{proof}

\begin{theorem} \label{preservation_theorem}
  Let $\kappa \leq \chi < \theta$ be infinite regular cardinals with 
  $\kappa^{<\kappa} = \kappa$ and let $\P := \Add(\kappa, \chi)$. 
  Suppose that $M \prec H(\theta)$ is a $\kappa^+$-guessing model such that $|M| = \kappa^+ \subseteq M$,
  $\P \in M$, and ${^{<\kappa}}M \subseteq M$. Then, in $V[\P]$, $M[\P]$ is a 
  $\kappa^+$-guessing model.
\end{theorem}

\begin{proof}
Since the proof is somewhat technical, we first give a broad overview of the proof strategy. We will fix an arbitrary cardinal $\lambda \in M$ and a $\P$-name $\dot{d}$ that is forced to be a subset of $\lambda$ that is $(\kappa^+, M[\P])$-approximated. This means that, for all $z \in M \cap \power_{\kappa^+} (\lambda)$, we can find a condition $p_z \in \P$ and a $\P$-name $\dot{d}_z \in M$ forced by $p_z$ to equal $\dot{d} \cap z$. By carefully stitching together certain of these $\P$-names $\dot{d}_z$ in a coherent way, we will construct a single condition $p_\emptyset \in \P$ and a (coding for) a $\P$-name $\tau$ with the following two crucial properties:
\begin{enumerate}
  \item $p_\emptyset \Vdash \dot{d} \cap M = \tau \cap M$;
  \item $\tau$ is $(\kappa^+,M)$-approximated.
\end{enumerate}

Item (2) above will imply that we can find $\sigma \in M$ such that $\sigma \cap M = \tau \cap M$. Then item (1) will be used to argue that, in $V[\P]$, we have $\sigma^G \cap M = \dot{d}^G \cap M$. Therefore, we will have shown that $\dot{d}$ is $M[\P]$-guessed and verified that $M[\P]$ is a $\kappa^+$-guessing model in $V[\P]$.

We now begin the actual proof: 
  Conditions in $\P$ are partial functions $p:\chi \rightarrow 2$ with $|p| < \kappa$, 
  ordered by reverse inclusion. Our cardinal arithmetic assumptions imply that forcing 
  with $\P$ preserves all cardinalities and cofinalities. Moreover, 
  $1_\P \Vdash ``M[\P] \cap \mathrm{On} = M \cap \mathrm{On}"$. Therefore, it suffices 
  to show that $M[\P]$ is forced to be a $\kappa^+$-guessing model for $\lambda$ for 
  every cardinal $\lambda \in M$.
  
  To this end, fix a cardinal $\lambda \in M$, a $\P$-name $\dot{d}$ for a subset of 
  $\lambda$, and a condition $p \in \P$ such that $p \Vdash ``\dot{d} \text{ is } 
  (\kappa^+, M[\P])\text{-approximated}"$. By Proposition \ref{unbounded_prop}, 
  $M$ is $\kappa^+$-internally unbounded. Therefore, we can write $\lambda \cap M$ 
  as $\bigcup_{\eta < \kappa^+} z_\eta$, where
  \begin{itemize}
    \item $\langle z_\eta \mid \eta < \kappa^+ \rangle$ is $\subseteq$-increasing;
    \item for all $\eta < \kappa^+$, we have $|z_\eta| \leq \kappa$ and 
    $z_\eta \in M$.
  \end{itemize}
  For each $\alpha < \lambda$, let $A_\alpha$ 
  be a maximal antichain of $\P$ below $p$ consisting of conditions deciding the 
  statement $\alpha \in \dot{d}$, and let $u_\alpha = \bigcup \{\dom{p} \mid p \in 
  A_\alpha\}$. Since $\P$ has the $\kappa^+$-cc, we know that $|u_\alpha| \leq \kappa$. 
  Moreover, for every $q \le p$, if $q$ decides the statement $\alpha \in
  \dot{d}$, then $q \restriction u_\alpha$ already decides the statement in the same way.
  
  For any $\P$-name $\dot{x}$ for a subset of $\lambda$, let 
  $\hat{\dot{x}}$ denote the function from $\P \times \lambda$ to $3$ such that, for all 
  $(q,\alpha) \in \P \times \lambda$, we have
  \begin{itemize}
    \item $\hat{\dot{x}}(q,\alpha) = 0$ if and only if $q \Vdash \alpha \in \dot{x}$;
    \item $\hat{\dot{x}}(q,\alpha) = 1$ if and only if $q \Vdash \alpha\notin\dot{x}$;
    \item $\hat{\dot{x}}(q,\alpha) = 2$ if and only if $q$ does not decide the statement 
    $\alpha \in \dot{x}$.
  \end{itemize}
  
  For each $\eta < \kappa^+$, using the fact that $p$ forces $\dot{d}$ to be 
  $(\kappa^+, M[\P])$-approximated, find $p_\eta \le p$ and a $\P$-name 
  $\dot{d}_\eta \in M$ such that $p_\eta \Vdash \dot{d} \cap z_\eta = 
  \dot{d}_\eta$. Since $\kappa^{<\kappa} = \kappa$, by passing to a cofinal 
  subsequence of $\langle p_\eta \mid \eta < \kappa^+ \rangle$ if necessary, 
  we may assume that
  \begin{itemize}
    \item $\{\dom{p_\eta} \mid \eta < \kappa^+\}$ forms a $\Delta$-system, with root $r$;
    \item there is a condition $p_\emptyset \in \P$ such that $p_\eta \restriction r = 
    p_\emptyset$ for all $\eta < \kappa^+$.
  \end{itemize}
  For each $\eta < \kappa^+$, let $s_\eta := \dom{p_\eta} \setminus r$. We will show that 
  $p_\emptyset$ forces $\dot{d}$ to be $M[\P]$-guessed. Since $p_\emptyset \le p$ and 
  $p$ was chosen arbitrarily, this suffices to prove the theorem.
  
  \begin{claim} \label{transfer_claim}
    Suppose that $\eta < \kappa^+$, $\alpha \in M \cap \lambda$, and $q \in \P \cap M$ are such that $\alpha \in 
    z_\eta$, $q \le p_\emptyset \cap M$, and $q \| p_\eta$. Then
	$\hat{\dot{d}}_\eta(q \cup (p_\eta \cap M), \alpha) = \hat{\dot{d}}(q \cup p_\eta, \alpha)$.    
  \end{claim}
  
  \begin{proof}
    We will prove that $\hat{\dot{d}}_\eta(q \cup (p_\eta \cap M), \alpha) = 0$ if and only if 
    $\hat{\dot{d}}(q \cup p_\eta, \alpha) = 0$. The proof of the rest of the claim is the same, 
    \emph{mutatis mutandis}. For the forward direction, suppose that $\hat{\dot{d}}_\eta(q \cup 
    (p_\eta \cap M), \alpha) = 0$. Then $q \cup (p_\eta \cap M) \Vdash \alpha \in \dot{d}_\eta$ 
    and $p_\eta \Vdash \dot{d} \cap z_\eta = d_\eta$, so, since $\alpha \in z_\eta$, it 
    follows that $q \cup p_\eta \Vdash \alpha \in \dot{d}$, and hence 
    $\hat{\dot{d}}(q \cup p_\eta, \alpha) = 0$.
    
    For the backward direction, suppose that $\hat{\dot{d}}_\eta(q \cup (p_\eta \cap M), \alpha) \neq 0$.
    Then we can find $q' \le q \cup (p_\eta \cap M)$ in $M$ 
    such that $q' \Vdash \alpha \notin \dot{d}_\eta$. Since $q' \in M$ and $q' \le 
    p_\eta \cap M$, it follows that $q' \| p_\eta$. Then $q' \cup p_\eta \Vdash \dot{d} 
    \cap z_\eta = \dot{d}_\eta$, and hence $q' \cup p_\eta \Vdash \alpha \notin \dot{d}$. 
    Since $q' \cup p_\eta \le q \cup p_\eta$, it follows that $q \cup p_\eta 
    \not\Vdash \alpha \in \dot{d}$, i.e., $\hat{\dot{d}}(q \cup p_\eta, \alpha) \neq 0$.
  \end{proof}
  
  For each $(q,\alpha) \in (\P \times \lambda) \cap M$ with $q \le p_\emptyset \cap M$, 
  let $B_{q,\alpha}$ be the set of all $\eta < \kappa^+$ such that $\alpha \in z_\eta$ 
  and $s_\eta \cap (u_\alpha \cup \dom{q}) = \emptyset$. Since $\langle s_\eta \mid 
  \eta < \kappa^+ \rangle$ is a sequence of pairwise disjoint sets, it follows that $|\kappa^+ \setminus 
  B_{q,\alpha}| \leq \kappa$. Also, for all $\eta \in B_{q,\alpha}$, since $q \in M$,
  $q \leq p_\emptyset \cap M$, and $s_\eta \cap \dom{q} = \emptyset$, we have $q \|p_\eta$.
  
  \begin{claim} \label{uniformity_claim}
    Suppose that $(q,\alpha) \in (\P \times \lambda) \cap M$ and $q \le p_\emptyset \cap M$.
    Then, for all $\eta \in B_{q,\alpha}$, we have 
    \[
      \hat{\dot{d}}_\eta(q \cup (p_\eta \cap M), \alpha) = \hat{\dot{d}}(q \cup p_\emptyset, \alpha).
    \]
  \end{claim}
  
  \begin{proof}
    Fix $\eta \in B_{q,\alpha}$.
    We will prove that $\hat{\dot{d}}_\eta(q \cup (p_\eta \cap M), \alpha) = 0$ if and only if 
    $\hat{\dot{d}}(q \cup p_\emptyset, \alpha) = 0$. The proof of the rest of the claim is the same.
    For the forward direction, suppose that $\hat{\dot{d}}_\eta(q \cup (p_\eta \cap M), \alpha) = 0$.
    Then, by Claim \ref{transfer_claim}, $\hat{\dot{d}}(q \cup p_\eta, \alpha) = 0$, i.e., 
    $q \cup p_\eta \Vdash \alpha \in \dot{d}$. It follows that 
    $(q \cup p_\eta) \restriction u_\alpha \Vdash \alpha \in \dot{d}$. But 
    $s_\eta \cap u_\alpha = \emptyset$, so $q \cup p_\emptyset \le (q \cup p_\eta) 
    \restriction u_\alpha$, i.e., $\hat{\dot{d}}(q \cup p_\emptyset, \alpha) = 0$. 
    
    For the backward direction, suppose that $\hat{\dot{d}}(q \cup p_\emptyset, \alpha) = 0$. 
    Then, \emph{a fortiori}, $\hat{\dot{d}}(q \cup p_\eta, \alpha) = 0$, so, by 
    Claim \ref{transfer_claim}, we have $\hat{\dot{d}}_\eta(q \cup (p_\eta \cap M), \alpha) = 0$, 
    as desired.
  \end{proof}
  
  Let $D := \{(q,\alpha) \in (\P \times \lambda) \cap M \mid q \le p_\emptyset \cap M\}$, 
  and define a function $\tau : D \rightarrow 3$ by letting $\tau(q,\alpha) := 
  \hat{\dot{d}}(q \cup p_\emptyset, \alpha)$ for all $(q,\alpha) \in D$.
  
  \begin{claim}
    $\tau$ is $(\kappa^+, M)$-approximated.
  \end{claim}
  
  \begin{proof}
    It suffices to show that, for every $y \in M$ with $|y| = \kappa$, we have $\tau \restriction 
    y \in M$. Fix such a $y$. We can assume that $y \subseteq \P \times \lambda$ and, for 
    all $(q,\alpha) \in y$, we have $q \leq p_\emptyset \cap M$. Since $|y| = \kappa$, we can 
    find $\eta \in \bigcap\{B_{q,\alpha} \mid (q,\alpha) \in y\}$. By Claim \ref{uniformity_claim}, 
    it follows that, for all $(q,\alpha) \in y$, we have $\tau(q,\alpha) = 
    \hat{\dot{d}}_\eta(q \cup (p_\eta \cap M), \alpha)$. Since $\hat{\dot{d}}_\eta$,
    $p_\eta \cap M$, and $y$ are in $M$, it follows that $\tau \restriction y$ is definable in $M$ 
    and is therefore an element of $M$.
  \end{proof}
  
  Since $M$ is a $\kappa^+$-guessing model, we can find $\sigma \in M$ such that $\sigma \cap M = 
  \tau \cap M = \tau$. By elementarity, $\sigma$ is a function from $\{q \in \P \mid q \le 
  p_\emptyset \cap M\} \times \lambda$ to $3$.
  
  Let $G \subseteq \P$ be a $V$-generic filter with $p_\emptyset \in G$. Let 
  $E := \{\eta < \kappa^+ \mid p_\eta \in G\}$. By a standard density argument, we have 
  $|E| = \kappa^+$. Let $d$ be the interpretation of $\dot{d}$ in $V[G]$, and, for all
  $\eta < \kappa^+$, let $d_\eta$ be the interpretation of $\dot{d}_\eta$. Let 
  \[
    e := \{\alpha < \lambda \mid \exists q \in G ~ [\sigma(q,\alpha) = 0]\}.
  \]
  Everything needed to define $e$ is in $M[G]$, so $e \in M[G]$. We will be done if we show that 
  $e \cap M[G] = d \cap M[G]$. We first prove two preliminary claims.
  
  \begin{claim} \label{evaluation_claim}
    Let $\alpha \in \lambda \cap M$.
    \begin{enumerate}
      \item $\alpha \in d$ if and only if there is $q \in G \cap M$ such that $\tau(q,\alpha) = 0$.
      \item $\alpha \notin d$ if and only if there is $q \in G \cap M$ such that $\tau(q,\alpha) = 1$.
    \end{enumerate}
  \end{claim}
  
  \begin{proof}
    We prove (1). The proof of (2) follows by a symmetric argument. For the forward direction, 
    suppose that $\alpha \in d$. Find $\eta \in E$ such that $\alpha \in z_\eta$ and 
    $s_\eta \cap u_\alpha = \emptyset$. Then $d \cap z_\eta = d_\eta$, so we can find 
    $q \in G \cap M$ such that $q \Vdash \alpha \in \dot{d}_\eta$. 
    It follows that $q \cup p_\eta \Vdash \alpha \in \dot{d}$, and therefore 
    that $(q \cup p_\eta) \restriction u_\alpha \Vdash \alpha \in \dot{d}$.
    Since $s_\eta \cap u_\alpha = \emptyset$, we have $q \cup p_\emptyset \le 
    (q \cup p_\eta) \restriction u_\alpha$. Therefore, $q \cup p_\emptyset \Vdash 
    \alpha \in \dot{d}$, and hence $\tau(q,\alpha) = 0$.
    
    For the backward direction, suppose that $q \in G \cap M$ and $\tau(q,\alpha) = 0$. 
    Find $\eta \in E \cap B_{q,\alpha}$. By Claim \ref{uniformity_claim}, we have 
    $\hat{\dot{d}}_\eta(q \cup (p_\eta \cap M),\alpha) = 0$. Since $q \cup p_\eta \in G$, 
    we therefore have $\alpha \in d_\eta$. Again since $p_\eta \in G$, we have $d \cap z_\eta 
    = d_\eta$, and hence $\alpha \in d$.
  \end{proof}
  
  \begin{claim} \label{nonsense_claim}
    For all $\alpha \in \lambda$ and $q_0, q_1 \in \P$ with $q_0, q_1 \le p_\emptyset \cap M$, 
    if $\sigma(q_0,\alpha) = 0$ and $\sigma(q_1,\alpha) = 1$, then $q_0$ and $q_1$ are 
    incompatible in $\P$.
  \end{claim}
  
  \begin{proof}
    This is immediate from the elementarity of $M$ and the fact that $\sigma \cap M = \tau \cap M$.
  \end{proof}
  
  We are now ready to show that $e \cap M[G] = d \cap M[G]$. Fix $\alpha \in \lambda \cap M$, and 
  suppose first that $\alpha \in d$. By Claim \ref{evaluation_claim}, there is $q \in G \cap M$ 
  such that $\tau(q,\alpha) = 0$. But then we also have $\sigma(q,\alpha) = 0$, so $\alpha \in e$. 
  For the other direction, suppose that $\alpha \notin d$. Again by Claim \ref{evaluation_claim}, 
  there is $q \in G \cap M$ such that $\tau(q,\alpha) = 1$, and hence $\sigma(q,\alpha) = 1$. 
  By Claim \ref{nonsense_claim}, there cannot be $q' \in G$ such that $\sigma(q',\alpha) = 0$, 
  and therefore $\alpha \notin e$.
\end{proof}

We immediately obtain the following corollary.

\begin{corollary}\label{cor:GMP}
  Suppose that $\kappa$ is a regular infinite cardinal such that $\kappa^{<\kappa} = \kappa$ and 
  $\GMP_{\kappa^{++}}$ holds. Then $\GMP_{\kappa^{++}}$ is preserved by adding any number of Cohen 
  subsets to $\kappa$.
\end{corollary}

\begin{proof}
  Fix $\chi \geq \kappa$, let $\P := \Add(\kappa, \chi)$, and let $\theta > \chi$ be a sufficiently large 
  regular cardinal. Let $\dot{C}$ be a $\P$-name for a club in $(\power_{\kappa^{++}} (H(\theta)))^{V[\P]}$.
  It suffices to prove that in $V[\P]$ there is forced to be a $\kappa^+$-guessing model $\dot{N} 
  \in \dot{C}$ such that ${^{<\kappa}}\dot{N} \subseteq \dot{N}$. To this end, find a $\kappa^+$-guessing 
  model $M \prec H(\theta)$ such that $\P, \dot{C} \in M$, 
  $|M| = \kappa^+$, and ${^{<\kappa}}M \subseteq M$. By Theorem \ref{preservation_theorem}, 
  $M[\P]$ is forced to be a guessing model in $V[\P]$. Moreover, $\dot{C} \cap M[\P]$ is forced 
  to be a directed subset of $\dot{C}$ of size $\kappa^+$ whose union is all of $M[\P]$. Since 
  $\dot{C}$ is forced to be a club in $(\power_{\kappa^{++}} (H(\theta)))^{V[\P]}$, it follows that 
  $M[\P]$ is forced to be in $\dot{C}$. Finally, since ${^{<\kappa}}M \subseteq M$, it is forced to 
  be the case that ${^{<\kappa}}M[\P] \subseteq M[\P]$. Therefore, $\GMP_{\kappa^{++}}$ continues 
  to hold in $V[\P]$.
\end{proof}

Note that in the previous proof, we showed that in $V[\P]$, $\GMP_{\kappa^{++}}(\theta)$ holds for all sufficiently large regular $\theta \geq \kappa^{++}$, whereas $\GMP_{\kappa^{++}}$ asserts this for all regular $\theta \geq \kappa^{++}$. These are obviously the same, since, if $\kappa^{++} \leq \theta \leq \theta'$ are regular cardinals and $M \prec H(\theta')$ is a $\kappa^+$-guessing model with $\theta \in M$, then $M \cap H(\theta)$ is a $\kappa^+$-guessing model.

\begin{corollary} \label{cor:PFA}
Assume $\PFA$ holds. Then both $\TP(\omega_2)$ and $\neg \wKH(\omega_1)$ are preserved by adding any number of Cohen subsets of $\omega$.
\end{corollary}

\begin{proof}
This follows from Fact \ref{PFA}, Lemma \ref{lm:GMPimplies}, and Corollary \ref{cor:GMP}.
\end{proof}

\brm
Note that even a single Cohen subset of $\omega$ destroys $\PFA$ because it adds an $\omega_1$-Suslin tree. So it is essential to isolate a principle weaker than $\PFA$ for our preservation result.
\erm

\section{Preservation of the negation of the weak Kurepa Hypothesis by centered forcings}\label{sec:centered}

\subsection{Centered forcings}

\begin{definition}\label{def:centered} 
Let $\P$ be a forcing and suppose $\kappa$ is a cardinal. We say that $\P$ is \emph{$\kappa^+$-centered} if $\P$ can be written as the union of a family $\set{\P_\alpha\sub \P}{\alpha <\kappa}$ such that for every $\alpha<\kappa$: \begin{equation}\label{eq:c1} \mbox{for every $p,q \in \P_\alpha$ there exists $r \in \P_\alpha$ with $r \le p,q$}.\end{equation} If $\kappa = \omega$, we say that $\P$ is \emph{$\sigma$-centered}.
\end{definition}

It follows that $\P$ can be written as a union of $\kappa$-many filters if we close each $\P_\alpha$ upwards. We require (\ref{eq:c1}) to ensure nice properties of the system $\ST$ defined in Section \ref{sec:der} (in particular the transitivity of $<_i$).

Some definitions of $\kappa^+$-centeredness require just the compatibility of the conditions, with a witness not necessarily in $\P_\alpha$. The condition (\ref{eq:c1}) in this case reads:
\begin{multline}
\label{eq:c2} \mbox{for every $n<\omega$ and every sequence $p_0, p_1, \ldots, p_{n-1}$}\\\mbox{of conditions in $\P_\alpha$ there exists $r \in \P$ with $r \le p_i$ for every $0 \le i < n$}.
\end{multline}

The conditions (\ref{eq:c1}) and (\ref{eq:c2}) are not in general equivalent (see Kunen \cite{Kunen:new}, before Exercise III.3.27), but the distinction is not so important for us because the common forcings such as the Cohen forcing and the Prikry forcings are all centered in the stronger sense of (\ref{eq:c1}). Also note that the conditions are equivalent for Boolean algebras: the definition (\ref{eq:c2}) means that each $\P_\alpha$ is a system with FIP (finite intersection property), and as such can be extended into a filter.

\subsection{Systems and derived systems}\label{sec:system}\label{sec:der}

Suppose $\P$ is a forcing notion. In order to show that certain objects cannot exist in a generic extension $V[\P]$ (such as a weak Kurepa tree), we will work in the ground model and work with a system derived from a $\P$-name for the object in question. We give the general definition of a \emph{system} here and discuss systems derived from names in Definition \ref{def:derived} below. We formulate the definition of the system to fit our purpose, which gives a slightly less general concept than the one introduced in \cite{MAGSH:succ}.

\begin{definition}\label{def:system}
Let $\kappa\le\lambda$ be cardinals and let $D\sub\lambda$ be unbounded in $\lambda$. For each $\alpha\in D$, let $S_\alpha\sub\se{\alpha}\times \kappa$ and let $S=\bigcup_{\alpha\in D} S_\alpha$.\footnote{The elements of $S$ are therefore ordered pairs of ordinals; if the ordinals are not important, we denote the pairs of ordinals by letters $x,y,\ldots$, etc.} Moreover, let $I$ be an index set of cardinality $\le \kappa$ and $\RR=\set{<_i}{i\in I }$ a collection of binary relations on $S$. We say that $\seqv{S,\RR}$ is a \emph{$(\kappa,\lambda)$-system} if the following hold:

\begin{enumerate}[(i)]
\item For each $i\in I$, $\alpha,\beta\in D$ and $\gamma,\delta<\kappa$; if $(\alpha,\gamma)<_i(\beta,\delta)$ then $\alpha<\beta$.
\item For each $i\in I$, $<_i$ is irreflexive and transitive.
\item For each $i\in I$, and $\alpha < \beta < \gamma$, $x \in S_\alpha$, $y \in S_\beta$ and $z \in S_\gamma$, if $x <_i z$ and $y <_i z$, then $x <_i y$. 
\item For all $\alpha<\beta$ there are $y\in S_{\beta}$ and $x\in S_{\alpha}$ and $i\in I$ such that $x<_i y$.
\end{enumerate}

We call a $(\kappa,\lambda)$-system $\seqv{S,\RR}$ a \emph{strong $(\kappa,\lambda)$-system} if the following strengthening of item (iv) holds:

\begin{enumerate} [(iv')]
\item For all $\alpha<\beta$ and for every $y\in S_{\beta}$ there are $x\in S_{\alpha}$ and $i\in I$ such that $x<_i y$.
\end{enumerate}

\end{definition}

If $\seqv{S,\RR}$ is a $(\kappa,\lambda)$-system, we say that the system has height $\lambda$ and width $\kappa$. We call $S_\alpha$ the $\alpha$-th level of $S$.

For the purposes of this paper we introduce the following definition:

\begin{definition}\label{def:well}
Suppose $\kappa \le \lambda$ are cardinals and let $\seqv{S,\RR}$ be a $(\kappa,\lambda)$-system. We call $\seqv{S,\RR}$ \emph{well-behaved} if $|\RR| < \kappa$, i.e.\ the number of relations is strictly smaller than the width of the system.
\end{definition}

A \emph{branch} of the system is a subset $B$ of $S$ such that for some $i\in I$, and for all $a \neq b\in B$, $a<_i b$ or $b<_i a$. A branch $B$ is \emph{cofinal} if for each $\alpha<\lambda$ there are $\beta\ge\alpha$ and $b\in B$ on level $\beta$.

Systems appear naturally when we wish to analyse in the ground model a $\P$-name $\dot{T}$ for a tree which is added by a forcing notion $\P$. We give the definition for the context in which we will use it (more general definitions are possible).

\begin{definition}\label{def:derived}
Assume $\kappa$ is a regular cardinal. Assume $\P$ is a $\kappa^+$-centered forcing notion; let $\P = \bigcup_{\alpha< \kappa}\P_\alpha$ where each $\P_\alpha$ is a filter. Assume further that $\P$ forces that $\dot{T}$ is a tree of height and size $\lambda$, where $\lambda \ge \kappa^+$ is regular. We assume that the domain of $T$ is $\lambda \x \lambda$, where the $\beta$-th level of $\dot{T}$ consists of pairs in $\{\beta\} \x \lambda$. We say that $\ST = \la \lambda \x \lambda,\RR \ra$ is a \emph{derived system} (with respect to $\P$ and $\dot{T}$) if it is a system with domain $\lambda \x \lambda$ which is equipped with binary relations $\RR = \set{<_\alpha}{\alpha < \kappa}$, where $$x <_\alpha y \iff (\exists p \in \P_\alpha)\; p \Vdash x <_{\dot{T}}y.$$
\end{definition}

Following the terminology of Definitions \ref{def:system} and \ref{def:well}, $\ST$ is a strong well-behaved $(\lambda,\lambda)$-system.

\subsection{A preservation theorem}

We can prove a little more general result for the negation of the weak Kurepa hypothesis at $\kappa^+$: we show that it is preserved over any model of $\GMP_{\kappa^{++}}$ by all $\kappa^+$-centered forcings. Note that $\Add(\kappa,\kappa^+)$ is $\kappa^+$-centered, and since a weak Kurepa tree at $\kappa^+$ has size only $\kappa^+$, Theorem \ref{th:2} implies that $\neg \wKH(\kappa^+)$ is preserved by adding any number of Cohen subsets of $\kappa$ (see Corollary \ref{cor:any}). Also recall that $\MA_{\omega_1}$ implies that every ccc forcing of $\omega_1$ is $\sigma$-centered, so our result implies that over models of $\GMP_{\omega_2} + \MA_{\omega_1}$, and hence also of $\PFA$, $\neg \wKH(\omega_1)$ is preserved by all ccc forcings of size $\omega_1$ (see Corollary \ref{cor:MA}).

\begin{theorem}\label{th:2}
$\GMP_{\kappa^{++}}$ implies that $\neg\wKH(\kappa^+)$ is preserved by any $\kappa^+$-centered forcing.
\end{theorem}

\begin{proof}
Suppose $\P = \bigcup_{i<\kappa}\P_i$ is a $\kappa^+$-centered forcing. Assume for contradiction that $\dot{T}$ is forced by the weakest condition in $\P$ to be a weak $\kappa^+$-Kurepa tree, and let $\ST$ be the derived system with respect to $\dot{T}$, as in Definition \ref{def:derived}. Let $M$ be a $\kappa^+$-guessing model $M \el H(\kappa^{++})$ of size $\kappa^+$ with $\ST \in M$. Since $\ST$ has domain $\kappa^+ \x \kappa^+$, the system with the relations is a subset of $M$.

Let us fix a sequence $\seq{\dot{b}_\alpha}{\alpha<\kappa^{++}}$ of $\P$-names such that \begin{equation} \label{eq:1} 1_\P \Vdash {``}\seq{\dot{b}_\alpha}{\alpha<\kappa^{++}} \mbox{ are pairwise distinct cofinal branches in $\dot{T}$.''}\end{equation}
Working in $V$, there must be some $i < \kappa$, such that for some $I \sub \kappa^{++}$ of size $\kappa^{++}$, and all $\alpha \in I$, there are cofinally many $x$ for which there are $p_x \in \P_i$ with $$p_x \Vdash x \in \dot{b}_\alpha.$$ Let us fix such an $i < \kappa$ and $I \sub \kappa^{++}$.

For each $\alpha \in I$, let us define $$B_\alpha = \set{x \in \ST}{(\exists p \in \P_i)\; p \Vdash x \in \dot{b}_\alpha}.$$ Note that $B_\alpha$ is a cofinal branch in $\ST$. We finish the proof by showing:
\begin{enumerate}[(i)]
\item For each $\alpha \in I$, $B_\alpha$ is an element of $M$.
\item For all $\alpha \neq \beta \in I$, $B_\alpha \neq B_\beta$.
\end{enumerate}

The items (i) and (ii) imply that $M$ has size at least $\kappa^{++}$, which is a contradiction.

With regard to (i), we will show that each $B_\alpha$ is $M$-approximated, and therefore is an element of $M$.\footnote{In general, if some set $d$ is $M$-approximated, and $M$ is guessing, there is some $e \in M$ such that $d \cap M = e \cap M$, where $d \neq e$ is possible. However, in the present case, since $\ST \sub M$, if $B_\alpha$ is $M$-approximated, and $M$ is guessing, then $B_\alpha \in M$.} Let us fix $\alpha \in I$ and $a \in M$ of size $\kappa$. We need to show that $B_\alpha \cap a$ is in $M$. Since the sytem $\ST$ has height $\kappa^+$, and therefore its cofinality in $M$ is $\kappa^+$, there is some $y \in B_\alpha \cap M$ which is above $B_\alpha \cap a$ in $<_i$. It follows that $$B_\alpha \cap a = \set{x \in a \cap \ST}{(\exists p \in \P_i)\; p \Vdash x \in \dot{b}_\alpha} =  \set{x \in a \cap \ST}{x <_i y}.$$ For the identity between the second and third set, fix $p' \in \P_i$ such that $p' \Vdash y \in \dot{b}_\alpha$, and note that if $p \in \P_i$ and $p \Vdash x \in \dot{b}_\alpha$, then the existence of a lower bound of $p,p'$ in $\P_i$ implies $x <_i y$; and conversely, if $p \Vdash x <_{\dot{T}} y$ for some $p \in \P_i$, then the existence of a lower bound in $\P_i$ implies that for some $r \in \P_i$, $r \Vdash x \in \dot{b}_\alpha$. Since the third expression determines a set in $M$ (because all parameters are in $M$), $B_\alpha \cap a$ is in $M$.

With regard to (ii): suppose for contradiction $B_\alpha = B_\beta$ for some $\alpha \neq \beta \in I$. Fix for every $x \in B_\alpha = B_\beta$ some conditions $p^\alpha_x$ and $p^\beta_x$ in $\P_i$ such that $$p^\alpha_x \Vdash x \in \dot{b}_\alpha \mbox{ and } p^\beta_x \Vdash x \in \dot{b}_\beta.$$
Let $p_x \in \P_i$ be some lower bound of $p_x^\alpha, p_x^\beta$. 

Suppose first that  $\set{p_x}{x \in B_\alpha = B_\beta}$ has size $\kappa$. Then there exists some $p$ such that $p = p_x$ for $\kappa^+$ many $x$. This $p$ forces $\dot{b}_\alpha = \dot{b}_\beta$, which contradicts (\ref{eq:1}).

Suppose now that $\set{p_x}{x \in B_\alpha = B_\beta}$ has size $\kappa^+$. The $\kappa^+$-cc of $\P$ implies that there is a condition $p$ which forces that $\dot{G}$ has an intersection with $\set{p_x}{x \in B_\alpha = B_\beta}$ of size $\kappa^+$. In particular, $p$ forces $\dot{b}_\alpha = \dot{b}_\beta$, which contradicts (\ref{eq:1}).
\end{proof}

Theorem \ref{th:2} gives an alternative proof that $\neg \wKH(\kappa^+)$ is preserved by adding any number of Cohen subsets of $\kappa$ over models of $\GMP_{\kappa^{++}}$.

\begin{corollary}\label{cor:any}
Over models of $\GMP_{\kappa^{++}}$, $\neg \wKH(\kappa^+)$ is preserved by adding any number of Cohen subsets of $\kappa$.
\end{corollary}

\begin{proof}
It is known that if $\kappa^{<\kappa} = \kappa$, then for every $\alpha \le 2^\kappa$, Cohen forcing $\Add(\kappa,\alpha)$ is $\kappa^+$-centered. We only need that $\Add(\kappa,\kappa^+)$ is $\kappa^+$-centered for our argument, but note that since $\GMP_{\kappa^{++}}$ implies $2^\kappa \ge\kappa^{++}$, longer Cohen forcings are $\kappa^+$-centered. 

Suppose for contradiction that $\Add(\kappa,\gamma)$ adds a weak Kurepa tree $T$ at $\kappa^+$ for some $\gamma$. Using the fact that $T$ has size $\kappa^+$, $T$ is added already by $\Add(\kappa,\kappa^+)$. By Theorem \ref{th:2}, $T$ has at most $\kappa^+$ cofinal branches in $V[\Add(\kappa,\kappa^+)]$. Since Cohen forcing is $\kappa^+$-Knaster, in particular its product with itself is $\kappa^+$-cc, it follows that Cohen forcing at $\kappa$ cannot add new cofinal branches to $T$ over $V[\Add(\kappa,\kappa^+)]$. This is a contradiction.
\end{proof}

If we additionally assume $\MA_{\omega_1}$, then ccc forcings are more well-behaved (for instance they are all Knaster). Moreover, if they have size at most $\omega_1$, they are even $\sigma$-centered (see \cite[Theorem 4.5]{Top:Weiss}), so we obtain the following Corollary.

\begin{corollary} \label{cor:MA}
$\neg \wKH(\omega_1)$ is preserved over models of $\GMP_{\omega_2} +  \MA_{\omega_1}$ by any ccc forcing of size $\omega_1$. In particular over models of $\PFA$, $\neg \wKH(\omega_1)$ is preserved by all ccc forcings of size $\omega_1$. 
\end{corollary}

\section{An application}\label{sec:appl}

In this section, we sketch an application of the indestructibility result for the negation of the weak Kurepa Hypothesis to provide a proof of the consistency of $\neg \wKH(\aleph_{\omega+1})$. Familiarity with lifting arguments, as they appear for instance in \cite{ABR:tree}, \cite{HS:ind} or \cite{8fold}, is assumed.

Suppose $\kappa < \lambda$ are regular cardinals with $\kappa^{<\kappa} = \kappa$ and $\lambda$ inaccessible. The Mitchell forcing $\M(\kappa,\lambda)$ can be written as $\Add(\kappa,\lambda)*\dot{Q}$ for some quotient forcing $\dot{Q}$ which is forced to be $\kappa^+$-distributive (see \cite{ABR:tree} for more details). If $G$ is $\M(\kappa,\lambda)$-generic, we write $G_0* G_1$ to denote the corresponding $\Add(\kappa,\lambda)*\dot{Q}$-generic.

\begin{theorem}\label{th:3}
Suppose $\kappa < \lambda$ are supercompact cardinals and $\kappa$ is Laver-indestructibly supercompact. Let $G$ be $\M(\kappa,\lambda)$-generic, where $\M(\kappa,\lambda)$ is the Mitchell forcing. In $V[G]$, 
$\kappa$ is supercompact, $2^\kappa = \kappa^{++}$, and $\GMP_{\kappa^{++}}$ holds. Let $\Q$ be the Prikry forcing with interleaved collapses which turns $\kappa$ into $\aleph_{\omega}$. Suppose $F$ is $\Q$-generic over $V[G]$. Then in $V[G][F]$, the negation of the weak Kurepa Hypothesis holds at $\aleph_{\omega+1}$.
\end{theorem}

\begin{proof}
For the properties of the model $V[G]$, see Fact \ref{f:GMP}. With regard to the forcing $\Q$: $\Q$ is defined with respect to some guiding generic whose existence follows from the facts that $\kappa$ is still $\kappa^+$-supercompact in $V[G]$, $2^\kappa = \kappa^{++}$ in $V[G]$, and $V[G]$ can be written as $V[G_0][G_1]$, where $G_0$ is $\Add(\kappa,\lambda)$-generic. See \cite[Lemma 4.1]{8fold} for more details. Since the compatibility of conditions in $\Q$ depends only on the stems, $\Q$ is $\kappa^+$-centered. Then the theorem follows by Theorem \ref{th:2}.
\end{proof}

\brm
By an argument using a quotient analysis in \cite[Lemma 4.6]{8fold}, the tree property holds at $\aleph_{\omega+2}$ in the model $V[G][F]$.
\erm

If $\kappa$ is not turned into $\aleph_\omega$, but only singularised (to an arbitrary cofinality), we can apply an indestructibility result also for the tree property, following our \cite{HS:ind}.

\begin{theorem}\label{th:4}
Suppose $\kappa < \lambda$ are supercompact cardinals and $\kappa$ is Laver-indestructibly supercompact. Let $G$ be $\M(\kappa,\lambda)$-generic, where $\M(\kappa,\lambda)$ is the Mitchell forcing. In $V[G] = V[G_0][G_1]$, $\kappa$ is supercompact, $2^\kappa = \kappa^{++}$, and $\GMP_{\kappa^{++}}$ holds. Let $\Q$ be the Prikry or Magidor forcing which turns $\kappa$ into a singular cardinal without collapsing any cardinals; choose $\Q$ to be an element of $V[G_0]$.\footnote{This is possible; see the construction in \cite{HS:u} for more details.} Suppose $F$ is $\Q$-generic over $V[G]$. Then in $V[G][F]$, the negation of the weak Kurepa Hypothesis holds at $\kappa^+$ and the tree property holds at $\kappa^{++}$.
\end{theorem}

\begin{proof}
The part regarding the negation of the weak Kurepa Hypothesis is as in Theorem \ref{th:3} (but it is easier since no guiding generic needs to be constructed), and the tree property holds because the forcing $\Q$ lives in $V[G_0]$, so the indestructibility result from \cite{HS:ind} applies.
\end{proof}

Let us summarize the arguments in this section succinctly as follows:

\begin{corollary}\label{cor:o}
Assuming the consistency of the existence of two supercompact cardinals, it is consistent that $\neg \wKH(\aleph_{\omega+1}) + \TP(\aleph_{\omega+2})$ holds.
\end{corollary}

\brm
The large cardinal assumption in Corollary \ref{cor:o} can be substantially reduced if we calculate more carefully. Firstly, it is easy to see that it suffices to start only with one supercompact cardinal and a weakly compact cardinal above it: for $\neg \wKH(\kappa^+)$ only $\GMP_{\kappa^{++}}(\kappa^{++})$ is required as we mentioned just before Remark \ref{rm:more}, and a weakly compact cardinal is sufficient for this, as is also for $\TP(\aleph_{\omega+2})$ in \cite{8fold}. Secondly, one can use an indestructibility result provable for strongness-type cardinals (see \cite{RH:Laver}) and then follow an argument from \cite{FHS2:large}: this way, the large cardinal assumptions are reduced  to an $H(\lambda^+)$-strong cardinal $\kappa$, where $\lambda > \kappa$ is weakly compact, and this is almost optimal.
\erm

\section{Open questions} \label{sec:open}

\begin{enumerate}[(1)]
\item Can we extend the preservation result for $\TP(\kappa^{++})$ to all $\kappa^+$-centered forcings? It seems that our method in Theorem \ref{th:2} does not generalize to the tree property. The syntactical difference between the two principles may be important here: while $\wKH(\omega_1)$ is a $\Sigma_1$ sentence, $\neg \TP(\omega_2)$ is a $\Sigma_2$ sentence. This is also relevant for question (\ref{q:zap}) mentioned below.

\medskip

\noindent Or more ultimately, can the preservation result under $\GMP_{\kappa^{++}}$ be extended to all $\kappa^+$-cc forcings? To our knowledge there is no known counterexample.

\medskip 

\item Is the assumption of $\GMP_{\kappa^{++}}$ necessary for the preservation results? To our knowledge it is still open whether, for instance, there may be a model which satisfies $\TP(\omega_2)$ and/or $\neg \wKH(\omega_1)$, and over which a single Cohen forcing at $\omega$ adds an $\omega_2$-Aronszajn tree and/or a weak Kurepa tree at $\omega_1$.

\medskip

\item \label{q:zap} Zapletal observed that $\MA_{\omega_2}$ implies that $\neg \wKH(\omega_1)$ is preserved by all ccc forcings: if $\P$ is a ccc forcing notion and $\dot{T}$ is a $\P$-name for a weak $\omega_1$-Kurepa tree, then with $\MA_{\omega_2}$ one can use $\dot{T}$ to define a weak Kurepa tree back in $V$. However, it is not known whether $\MA_{\omega_2}$ is consistent with $\neg \wKH(\omega_1)$.

\medskip

\noindent It is known that $\MA_{\omega_2} + \neg \KH(\omega_1)$ is consistent due to a result of Jensen and Schlechta \cite{JS:k} who showed that starting with $\CH$, in the Levy collapse by countable conditions of a Mahlo cardinal $\kappa$ to $\omega_2$, $\neg \KH(\omega_1)$ is preserved by all ccc forcings. In particular, it is is preserved by any ccc finite-support iteration which forces $\MA_{\omega_2}$ over this model. (We do not know an easier proof of the consistency of $\MA_{\omega_2} + \neg \KH(\omega_1)$.)

\medskip

\noindent We can ask whether there is an analogous construction for $\neg \wKH(\omega_1)$: can one collapse a large cardinal $\kappa$ to become $\omega_2$ with $2^\omega \ge \omega_2$ and obtain a model over which $\neg \wKH(\omega_1)$ is preserved by all ccc forcings? Note that this line of argument cannot be used with $\PFA$ because $\MA_{\omega_2}$ implies $2^\omega > \omega_2$. More generally, if $\MA_{\theta}$ holds for a regular $\theta$, then we can only argue that if $\dot{T}$ is a $\P$-name for a weak Kurepa tree at $\omega_1$ with at least $\theta$ cofinal branches, there is one in $V$ with at least $\theta$ cofinal branches.

\medskip

\item Is it possible to strengthen Theorem \ref{th:2} along the lines of Corollary \ref{cor:GMP} and show that not only $\neg \wKH(\kappa^+)$, but the principle $\GMP_{\kappa^{++}}$ itself is preserved by all $\kappa^+$-centered forcings? Or perhaps by all $\kappa^+$-cc forcings?

\end{enumerate}

\bibliographystyle{amsplain}
\bibliography{mybiblio}
\end{document}